\documentclass[10pt, twoside]{amsart}

\usepackage{amsmath,amssymb}

\newtheorem{thm}{Theorem}[section]
\newtheorem{prop}[thm]{Proposition}

\newtheorem{cor}[thm]{Corollary}

\theoremstyle{definition}

\newtheorem{definition}[thm]{Definition}
\newtheorem{example}[thm]{Example}

\theoremstyle{remark}

\newtheorem{remark}[thm]{Remark}

\numberwithin{equation}{section}

\newcommand{\F}{{\mathbb F}}
\newcommand{\N}{{\mathbb N}}

\newcommand{\p}{{\mathfrak p}}

\newcommand{\m}{{\mathfrak m}}
\newcommand{\V}{{\mathcal V}}
\newcommand{\gammasep}{{\gamma}_{\operatorname{sep}}}
\newcommand{\Vsep}{{\mathcal V}_{\operatorname{sep}}}

\DeclareMathOperator{\codim}{codim}
\DeclareMathOperator{\GL}{GL}
\DeclareMathOperator{\cmdef}{cmdef}
\DeclareMathOperator{\Spec}{Spec}
\DeclareMathOperator{\depth}{depth}
\DeclareMathOperator{\height}{ht}

\usepackage[baseurl={http://www-m11.ma.tum.de/reimers},
            pdftitle={Polynomial Separating Algebras and Reflection Groups},
            pdfauthor={Fabian Reimers},
			colorlinks=true,
			linkcolor=red,
			citecolor=blue         
            ]{hyperref}

\begin{document}

\title{Polynomial Separating Algebras and Reflection Groups}

\author{Fabian Reimers}
\address{Technische Universi\"at M\"unchen, Zentrum Mathematik - M11, 
Boltzmannstr.~3, 85748 Garching, Germany}
\email{reimers@ma.tum.de}

\begin{abstract}
This note considers a finite algebraic group $G$ acting on an affine variety $X$ by automorphisms. Results of Dufresne on polynomial separating algebras for linear representations of $G$ 
are extended to this situation.
For that purpose, we show that the 
Cohen-Macaulay defect of 
a certain ring
is greater than or equal to the minimal number $k$ such that the group is generated by $(k+1)$-reflections.
Under certain rather mild assumptions on $X$ and $G$ we deduce that a separating set of invariants of the smallest possible size $n = \dim(X)$ can exist only for reflection groups. 
\end{abstract}

\maketitle

%\tableofcontents

%%%%%%%%%%%%%%%%%%%%%%%%%%%%%%%%%%%%%%%%%%%%%%%%%%%%%%%%%%%%%%%%%%%%%%
\section*{Introduction}
%%%%%%%%%%%%%%%%%%%%%%%%%%%%%%%%%%%%%%%%%%%%%%%%%%%%%%%%%%%%%%%%%%%%%%

In the invariant theory of finite groups it has long been known that invariant rings with the best structural properties, i.e.~isomorphic to polynomial rings, can exist only for reflection groups (Shephard and Todd \cite{shephard1954finite}, Chevalley \cite{chevalley1955invariants} and Serre \cite{serre1968groupes}). In the non-modular case the converse also holds.

Subsets of the invariant ring that have the same capability of separating orbits
% as the whole invariant ring 
have turned out to be in many cases better behaved than generating sets. For example, there always exists a finite separating subset (Derksen and Kemper \cite{derksen2002computational}). We refer to \cite{kemper2009separating} for a nice and detailed survey of separating invariants.

Recently, Dufresne \cite{dufresne2009separating} proved the following generalization of the Theorem of Shephard, Todd, Chevalley and Serre: If $V$ is an $n$-dimensional linear representation of $G$ and there exists a separating set of invariants of size $n$, then the group is generated by reflections. In addition, an example was given where the invariant ring is not polynomial, but a smaller separating subalgebra is.

The separating variety is an object that naturally appears when studying separating invariants. For finite groups it is just the graph of the action (cf. Prop. \ref{PropositionIrreducibleComponentsOfVsep}).
An important step in Dufresne's proof is her discovery how the connectedness in codimension 1 of the separating variety
implies that the group is a reflection group. The proof relies on Hartshorne's Connectedness Theorem. In this short note our aim is to extend these methods from linear actions to actions on affine varieties and simultaneously from reflections to $k$-reflections where $k$ not necessarily equals $1$. We want to assume as little as necessary for the variety $X$ and the group action. For that purpose, we achieve an if-and-only-statement about the connectedness in a certain codimension of the separating variety, which is given as Theorem \ref{TheoremVsepConnectedInCodimIfAndOnlyIfGgenerated} of this note.

Then we assume that $X$ is connected and combine Theorem \ref{TheoremVsepConnectedInCodimIfAndOnlyIfGgenerated} with a version of Hartshorne's Connectedness Theorem that uses the Cohen-Macaulay defect. From that we conclude that the Cohen-Macaulay defect (or, more precisely, the set-theoretical Cohen-Macaulay defect) of the separating variety is greater than or equal to the minimal number $k$ such that $G$ is generated by $(k+1)$-reflections (see Theorem \ref{MainTheorem}).
In Corollary \ref{MainCorollary} we extend Dufresne's result to Cohen-Macaulay varieties and groups that are generated by elements having a fixed point. In the subsequent examples we show that these hypotheses on $X$ and $G$ cannot be dropped.

This paper is part of my ongoing PhD project. I would like to thank my thesis supervisor, Gregor Kemper, for his encouragement and permanent support.

%%%%%%%%%%%%%%%%%%%%%%%%%%%%%%%%%%%%%%%%%%%%%%%%%%%%%%%%%%%%%%%%%%%%%%
\section{The Cohen-Macaulay Defect And Connectedness}
%%%%%%%%%%%%%%%%%%%%%%%%%%%%%%%%%%%%%%%%%%%%%%%%%%%%%%%%%%%%%%%%%%%%%%

Let us start by fixing some notation. We will write $K$ for the base field, which is assumed to be algebraically closed. Furthermore, $G$ is a finite algebraic group and $X$ an affine variety (both defined over $K$) on which $G$ acts by automorphisms. We will write $n$ for the dimension of $X$ throughout this paper.
The group action on $X$ induces an action on its coordinate ring $K[X]$. A subset $S \subseteq K[X]^G$ of the invariant ring is called separating if it suffices the following property: If there exist $x, y \in X$ and $f \in K[X]^G$ with $f(x) \neq f(y)$, then there exists $g \in S$ with $g(x) \neq g(y)$. Following Kemper \cite{kemper2009separating}, we write $\gammasep$ for the smallest natural number $m$ such that there exists a separating subset of size $m$.
%with the following property for all $x, y \in X$:
%\[ g(x) = g(y) \,\, \forall g \in S\,\, \Rightarrow \,\,f(x) = f(y) \,\,\forall f \in K[X]^G.\]
%For a finite group, there is always a geometric quotient. Hence, all orbits can be separated by invariants.

The separating variety $\Vsep$ is the following subvariety of $X \times X$:
\[ \Vsep = \{ (x,y) \in X \times X \mid f(x) = f(y)  \text{ for every } f \in K[X]^G \}.\]
We see that a set of invariants $S \subseteq K[X]^G$ is separating if and only if the ideal in $K[X] \otimes_K K[X]$ generated by $g \otimes 1 - 1 \otimes g$ with $g \in S$ defines $\Vsep$ as a variety, i.e. has the same radical as $\left( f \otimes 1 - 1 \otimes f \mid f \in K[X]^G \right)$.

Now we recall the definition of connectedness in a certain codimension (see \cite{hartshorne1962complete}). For $k \in \N_0$,
a noetherian topological space $Y$ is called
connected in codimension $k$ if it satisfies one of the following two equivalent properties:
\begin{enumerate}
\item for all closed subsets $A \subseteq Y$ with $\codim_Y(A) > k$, the space $Y \setminus A$ is connected,
\item for all irreducible components $Y'$ and $Y''$ of $Y$ there exists a finite sequence $Y_0, \, \ldots, \, Y_r$ of irreducible components of $Y$ with $Y_0 = Y'$, \, $Y_r = Y''$ and $\codim_Y(Y_i \cap Y_{i+1}) \leq k$.
\end{enumerate}
Of course, being connected in codimension $k$ implies being connected in codimension $k+1$.
So, we have a chain of properties of $Y$, starting from $Y$ being connected in codimension $0$, which is equivalent to $Y$ being irreducible, leading to $Y$ being connected in codimension $\dim(Y)$, which is equivalent to $Y$ simply being connected.

We will state Hartshorne's Connectedness Theorem in the form that we want to use.

\begin{thm}\label{TheoremHartshornesConnectedness} 
Let $R$ be a noetherian ring and $k \in \N_{0}$. We assume, that $\Spec(R)$ is connected and that for all $\p \in \Spec(R)$ with $\dim(R_{\p}) > k$ we have $\depth(R_{\p}) \geq 2$. Then $\Spec(R)$ is connected in codimension $k$. 
\end{thm}

\begin{proof}
In \cite[Corollary 2.3]{hartshorne1962complete} it was shown that $\Spec(R)$ is \emph{locally} connected in codimension k under this hypothesis. If we require $\Spec(R)$ to be connected, this implies that $\Spec(R)$ is connected in codimension $k$ (see \cite[Remark 1.3.2]{hartshorne1962complete}).
\end{proof}

For a noetherian local Ring $(R,\m)$ the Cohen-Macaulay defect is defined as
\[ \cmdef(R) := \dim(R) - \depth(R) \in \N_0.\] 
More generally, for a noetherian ring $R$ the Cohen-Macaulay defect is
\[\cmdef(R) := \sup\, \{\, \cmdef(R_{\p}) \mid \p \in \Spec(R) \,\} \in \N_0 \cup \{\infty\}, \]
which can be shown to be consistent with our definition of $\cmdef(R)$ for a local ring.

%Let us deduce a corollary from Hartshorne's Connectedness Theorem, by applying the Cohen-Macaulay defect.
Let us express a weaker version of Hartshorne's Connectedness Theorem in terms of the Cohen-Macaulay defect.

\begin{cor}\label{CorollaryHartshornesConnectednessWithCMDEF}
Let $R$ be a noetherian ring. We assume that $\Spec(R)$ is connected and that $k := \cmdef(R)$ is finite. Then $\Spec(R)$ is connected in codimension $k + 1$. 
\end{cor}

%%%%%%%%%%%%%%%%%%%%%%%%%%%%%%%%%%%%%%%%%%%%%%%%%%%%%%%%%%%%%%%%%%%%%%
\section{Connectedness of $\Vsep$}
%%%%%%%%%%%%%%%%%%%%%%%%%%%%%%%%%%%%%%%%%%%%%%%%%%%%%%%%%%%%%%%%%%%%%%

We want to study the separating variety now. Starting with its irreducible components, we will precisely see what its connectedness in codimension $k$ means for $X$ and $G$.

\begin{prop}\label{PropositionIrreducibleComponentsOfVsep}
Let $X = \bigcup_{i=1}^r X_i$ be decomposed into its irreducible components $X_i$. Then for all $i$ and for all $\sigma \in G$
\[ H_{\sigma,i} := \{ (x, \sigma x) \mid x \in X_i\} \subseteq X \times X\]
is an irreducible component of $\Vsep$, and $\Vsep$ is the union of all $H_{\sigma,i}$. 
%\[ \Vsep = \bigcup_{\sigma \in G}\bigcup_{i=1}^r H_{\sigma,i}.\]
\end{prop}

\begin{proof}
Since $G$ is finite, we know that the invariants separate the orbits (see \cite[Section 2.3]{derksen2002computational}). Therefore, two points $x, y \in X$ lie in the same orbit if and only if $(x,y) \in \Vsep$. Hence,
$\Vsep$ is the union of all $H_{\sigma} := \{ (x, \sigma x) \mid x \in X\}$. Each $H_{\sigma}$ is an affine variety isomorphic to $X$, so it decomposes as
\[ H_{\sigma} = \bigcup_{i=1}^r H_{\sigma,i} \]
into its irreducible components.
\end{proof}

\begin{remark}\label{RemarkVsepEqudimensionalToo}
With the notation of Proposition \ref{PropositionIrreducibleComponentsOfVsep} we also see:
\begin{enumerate}
\item[(a)]
Each $H_{\sigma,i}$ is isomorphic to $X_i$. In particular, the separating variety has the same dimension as $X$. 
\item[(b)] The formula
$ \codim_{X \times X}(\Vsep) = \min\limits_{i = 1..r} \dim(X_i) $ 
holds. In particular, if $X$ is equidimensional, then we have $\codim_{X \times X}(\Vsep) = n$.
\end{enumerate}
\end{remark}

\begin{definition}
For $k \in \N_0$ we call an element $\sigma \in G$ a $k$-reflection if $X^{\sigma}$ has codimension at most $k$ in $X$. For $k = 1$ we simply say that $\sigma$ is a reflection.
\end{definition}

\begin{thm}\label{TheoremVsepConnectedInCodimIfAndOnlyIfGgenerated}
Let $k \in \N_0$. Then the separating variety $\Vsep$ is connected in codimension $k$ if and only if $X$ is connected in codimension $k$ and $G$ is generated by $k$-reflections.
\end{thm}

\begin{proof}
Again, let $X = \bigcup_{i=1}^r X_i$ be decomposed into its irreducible components $X_i$, which leads to the components $H_{\sigma,i}$ of $\Vsep$ as seen in Proposition \ref{PropositionIrreducibleComponentsOfVsep}. First, we want to look at the intersection of two components of $\Vsep$ and see which codimension arises. For $\sigma, \, \tau \in G$ and $i, \, j$ we have
\[ H_{\sigma, i} \cap H_{\tau, j} = \{ (x,y) \mid x \in X_i \cap X_j, \, y = \sigma x = \tau x \} \simeq (X_i \cap X_j)^{\tau^{-1}\sigma}.\]
We know from Remark \ref{RemarkVsepEqudimensionalToo} that $\dim(X) = n = \dim(\Vsep)$. In addition, we get
\begin{equation}\label{EquationCodimensionIntersection1} \codim_{\Vsep}(H_{\sigma, i} \cap H_{\tau, j}) = \codim_X((X_i \cap X_j)^{\tau^{-1}\sigma}).\end{equation}
Suppose $\Vsep$ is connected in codimension $k$. By assumption, for all $\sigma \in G$ and $i, j$ there exists a sequence of irreducible components $H_{\sigma_0, i_0},\, \dots,\, H_{\sigma_s, i_s}$ of $\Vsep$ with $i_0 = i,\, i_s = j,\, \sigma_0 = \iota$ (the neutral element of $G$), $\sigma_s = \sigma$ and 
\begin{equation}\label{EquationCodimensionIntersection2}\codim_{\Vsep}\left(H_{\sigma_l, i_l} \cap H_{\sigma_{l+1}, i_{l+1}}\right) \leq k  \quad \text{ for all } l. \end{equation}
Putting (\ref{EquationCodimensionIntersection1}) and (\ref{EquationCodimensionIntersection2}) together leads to the inequality \begin{equation}\label{EquationCodimensionIntersection3}\codim_{X}\left(X_{i_l} \cap X_{i_{l+1}}\right)^{\sigma_{l}^{-1} \sigma_{l+1}}) \leq k  \quad \text{ for all } l. \end{equation}
In particular, (\ref{EquationCodimensionIntersection3}) shows that $X_{i_l} \cap X_{i_{l+1}}$ has codimension $\leq k$. So we have a sequence of irreducible components from $X_{i_0} = X_i$ to $X_{i_s} = X_j$ that intersect in codimension $\leq k$, i.e.~$X$ is connected in codimension $k$.\\
Moreover, (\ref{EquationCodimensionIntersection3}) implies that all $X^{\sigma_{l}^{-1} \sigma_{l+1}}$ have codimension $\leq k$, i.e. each $\sigma_{l}^{-1} \sigma_{l+1}$ is a $k$-reflection. Using $\sigma_0 = \iota$ and $\sigma_s = \sigma$ we can write 
\[\sigma = \sigma_0^{-1} \sigma_s = (\sigma_0^{-1} \sigma_1) \cdot (\sigma_1^{-1} \sigma_2) \cdot \dots \cdot (\sigma_{s-1}^{-1} \sigma_s)\]
as a product of $k$-reflections.\\
So, we have proven the only-if-part by simply splitting \ref{EquationCodimensionIntersection3} into two weaker conclusions. It may therefore be surprising that the converse holds as well.\\
To prove it, let us start with $i,\, j$ and a sequence of components $X_{i_0},\, \dots,\, X_{i_s}$ with $X_i = X_{i_0}, \, X_j = X_{i_s}$ and $\codim_X(\left(X_{i_l} \cap X_{i_{l+1}}\right) \leq k$. Consequently, for $\sigma \in G$ we know from (\ref{EquationCodimensionIntersection1}), that all $H_{\sigma, i_l} \cap H_{\sigma, i_{l+1}}$ have codimension $\leq k$. So we already have a sequence from $H_{\sigma, i}$ to $H_{\sigma, j}$ as desired.\\
Now take two elements $\sigma', \, \sigma'' \in G$. By assumption, there exist $k$-reflections $\tau_1,\, \ldots,\, \tau_s \in G$ with $(\sigma')^{-1} \sigma'' = \tau_1 \cdot \ldots \cdot \tau_s$.
Since
\[\min_{j=0..r}\codim_X(X_j^{\tau_l}) = \codim_X(X^{\tau_l}) \leq k,\]
for each $\tau_l$ there exists an $i_l$ such that
\begin{equation}\label{EquationCodimensionIntersection4} \codim_X(X_{i_l}^{\tau_l}) \leq k. \end{equation}
If we write $\sigma_0 := \sigma'$ and $\sigma_l := \sigma_{l-1} \tau_l$ for $l = 1..s$, then
\[ \sigma_s = \sigma_0 \cdot \tau_1 \cdot \ldots \cdot \tau_s = \sigma' \cdot (\sigma'^{-1} \cdot \sigma'') = \sigma''.\]
It follows from (\ref{EquationCodimensionIntersection1}) together with (\ref{EquationCodimensionIntersection4}) that
\[\codim_{\Vsep}(H_{\sigma_{l-1}, i_l} \cap H_{\sigma_{l}, i_l}) = \codim_{X}((X_{i_l})^{\sigma_{l-1}^{-1} \sigma_l}) =  \codim_X((X_{i_l})^{\tau_l}) \leq k.\]
We already saw how to construct a sequence of components from every $H_{\sigma, i_l}$ to $H_{\sigma, i_{l+1}}$ as desired. Putting these together, for all $i,\, j$ we can construct a sequence
\[H_{\sigma_0, i},\, \ldots ,\, H_{\sigma_0, i_1}, H_{\sigma_1, i_1},\, \ldots ,\, H_{\sigma_1, i_2}, H_{\sigma_2, i_2},\, \ldots ,\, H_{\sigma_s, i_s},\, \ldots ,\, H_{\sigma_s, j},\]
from $H_{\sigma', i}$ to $H_{\sigma'', j}$ such that two successive components intersect in codimension $\leq k$.
\end{proof}

Since we need $\Vsep$ to be connected in the proof of our main Theorem \ref{MainTheorem}, we specialize Theorem \ref{TheoremVsepConnectedInCodimIfAndOnlyIfGgenerated} to the case $k = n = \dim(X)$.

\begin{cor} \label{CorollaryVsepConnectedIfAndOnlyIf}
%Let $X$ be equdimensional. Then
The separating variety $\Vsep$ is connected if and only if $X$ is connected and $G$ is generated by elements having a fixed point.
\end{cor}

%%%%%%%%%%%%%%%%%%%%%%%%%%%%%%%%%%%%%%%%%%%%%%%%%%%%%%%%%%%%%%%%%%%%%%
\section{Generated by $k$-Reflections}
%%%%%%%%%%%%%%%%%%%%%%%%%%%%%%%%%%%%%%%%%%%%%%%%%%%%%%%%%%%%%%%%%%%%%%

We will now combine Corollary \ref{CorollaryHartshornesConnectednessWithCMDEF} with our results on $\Vsep$. Let us fix some more notation, which we will also use in our examples. 
We write $R$ for the ring $K[X] \otimes_K K[X]$ and $I$ for the ideal in $R$ generated by $f \otimes 1 - 1 \otimes f$ with $f \in K[X]^G$. So $K[\Vsep] = R / \sqrt{I}$ is the coordinate ring of the separating variety.

\begin{thm}\label{MainTheorem} Let $X$ be connected. 
We assume furthermore that $G$ is generated by elements having a fixed point. Define
\[ k := \min\, \{ \cmdef(R/J) \mid J \subseteq R \text{ an ideal with } \sqrt{J} = \sqrt{I} \} .\] 
Then $G$ is generated by $(k+1)$-reflections.
\end{thm}

\begin{proof}
Our assumptions about $X$ and the action of $G$ imply that $\Vsep$ is connected (see Corollary \ref{CorollaryVsepConnectedIfAndOnlyIf}). Let $J$ be an ideal in $R$ with $\sqrt{J} = \sqrt{I}$ and $k = \cmdef(R/J)$.
Corollary \ref{CorollaryHartshornesConnectednessWithCMDEF} tells us now, that $\Spec(K[\Vsep]) \simeq \Spec(R/J)$ is connected in codimension $k+1$. Of course, it is equivalent to say that $\Vsep$ is connected in codimension $k+1$. Therefore, by Theorem \ref{TheoremVsepConnectedInCodimIfAndOnlyIfGgenerated}, 
$G$ is generated by $(k+1)$-reflections.
\end{proof}

In general, $I$ need not be radical, and neither $I$ nor $\sqrt{I}$ must have the smallest 
Cohen-Macaulay defect among all ideals $J \subseteq R$ with $\sqrt{J} = \sqrt{I}$ (see Example \ref{ExampleDifferentCMDEF}). Therefore, the number $k$ in Theorem \ref{MainTheorem} need not be the Cohen-Macaulay defect of $\Vsep$. Since an ideal is called set-theoretically Cohen-Macaulay if there exists a Cohen-Macaulay ideal with the same radical (cf. \cite{singh2007local}), we propose to call this number the set-theoretical Cohen-Macaulay defect of $\Vsep$. 

To the best of my knowledge, no algorithm is known for computing the set-theoretical Cohen-Macaulay defect of $R/I$. This might be the reason why I could not find an example in which this number is not the minimal $m$ such that $G$ is generated by $(m+1)$-reflections. However, there are surprisingly many examples (like Example \ref{ExampleDifferentCMDEF}) in which these two numbers coincide.

\begin{cor}\label{MainCorollary} Let $X$ be connected and Cohen-Macaulay. We assume furthermore that $G$ is generated by elements having a fixed point. If $\gammasep = n$, then $G$ is generated by reflections.
\end{cor}

\begin{proof}
Since $X$ is Cohen-Macaulay, we also know that $X \times X$ is Cohen-Macaulay. We use \cite[Theorem 2.1]{bouchiba2002tensor} as a reference for that. In addition, $X$ is connected, so $X$ and $X \times X$ are  also equidimensional, since local Cohen-Macaulay rings are equidimensional (\cite[Corollary 18.11]{eisenbud1995commutative}).

Now let $\{ f_1,\, \ldots ,\, f_n \}$ be a set of separating invariants. Based on this separating set we define the ideal
\[ J := \left( f_1 \otimes 1 - 1 \otimes f_1,\, \ldots ,\, f_n \otimes 1 - 1 \otimes f_n\right) \subseteq R, \]
which has the same radical as $I$. Hence we have
\[ \height(J) = \height(\sqrt{I}) = \codim_{X \times X}(\Vsep) =  n,\]
by Remark \ref{RemarkVsepEqudimensionalToo}.
This tells us that $J$ is a complete intersection ideal. Since $R$ is Cohen-Macaulay, we know, that $R/J$ is Cohen-Macaulay as well (see \cite[Proposition 18.13]{eisenbud1995commutative}). Now we can use Theorem \ref{MainTheorem} with $k = 0$.

\end{proof}

\begin{remark} Of course, the assumptions on $X$ and $G$ in Theorem \ref{MainTheorem} and Corollary \ref{MainCorollary} are satisfied if $X = V$ is a linear representation of $G$.
\end{remark}

Dufresne \cite{dufresne2009separating} gave an example of a representation for which the invariant ring is not a polynomial ring, but still $\gammasep$ equals $n$. This suggested that the choice of $J$ in Theorem \ref{MainTheorem} matters. The following example is constructed of the same type and results in various Cohen-Macaulay defects. 
It is taken from Kemper's et al. \cite{kemper2001database} database of invariant rings.

\begin{example}\cite[ID 10253]{kemper2001database}\label{ExampleDifferentCMDEF}
Let the base field $K$ be of characteristic $2$. As always, $K$ is assumed to be algebraically closed. We look at the following subgroup, isomorphic to $C_2 \times C_2 \times C_2$, of $\GL_4(K)$: 
\[ G := \langle \begin{pmatrix}1&0&0&0\\0&1&0&0\\0&0&1&0\\0&0&1&1\end{pmatrix},\, \begin{pmatrix}1&0&0&0\\1&1&0&0\\0&0&1&0\\0&0&0&1\end{pmatrix},\,\begin{pmatrix}1&0&0&0\\1&1&1&0\\0&0&1&0\\1&0&1&1\end{pmatrix} \rangle \subseteq \GL_4(K).\]
Its natural action on $V = K^4$ is generated by reflections.
Using the computer algebra system \textsc{magma} \cite{bosma1997magma}, we have computed the primary invariants
\begin{eqnarray*}
f_1 &:=& x_1,\\
f_2 &:=& x_3,\\
f_3 &:=& x_1^2 x_3 x_4 + x_1^2 x_4^2 + x_1 x_3^2 x_4 + x_1 x_3 x_4^2 + x_3^2 x_4^2 + x_4^4,\\
f_4 &:=& x_1^3 x_2 + x_1 x_2 x_3^2 + x_1 x_3^2 x_4 + x_1 x_3 x_4^2 + x_2^4 + x_2^2 x_3^2 + x_3^3 x_4 + x_3^2 x_4^2
\end{eqnarray*}
and a secondary invariant
\[h  := x_1^2 x_2 + x_1 x_2^2 + x_3^2 x_4 + x_3 x_4^2.\]
Hence, the invariant ring 
\[ K[V]^G = K[x_1,\,x_2,\,x_3,\,x_4]^G = K[f_1,\,f_2,\,f_3,\,f_4,\,h] \]
is not a polynomial ring. Between our generating invariants, there is the relation
\[ f_1^3 h + f_1^2 f_3 + f_1 f_2^2 h + f_2^2 f_4 + h^2 = 0.\]
If we define $g_3 := f_1 h + f_3$ and $g_4 := f_1 h + f_4$, then we get $h^2 = f_1^2 g_3 + f_2^2 g_4$. Therefore, $K[V]^G$ lies in the purely inseparable closure of $A := K[f_1,\,f_2,\,g_3,\,g_4]$ in $K[V]^G$, which shows that $\{ f_1,\,f_2,\,g_3,\,g_4 \}$ is separating.
With the notation as above, the invariant ring defines the ideal $I$ in $R$, which is not a radical ideal 
in this example. Let $J$ be the ideal in $R$ generated by $g \otimes 1 - 1 \otimes g$ with $g \in A$. Using the graded version of the Auslander-Buchsbaum formula, we calculated the following Cohen-Macaulay defects with \textsc{magma}:
\[ \cmdef(R/I) = 2, \quad \cmdef(R/\sqrt{I}) = 1, \quad \cmdef(R/J) = 0.\]
Of course, $\cmdef(R/J) = 0$ is not surprising, as it was used in Corollary \ref{MainCorollary}.
\end{example}

The following example shows that the assumption that $G$ has fixed points cannot be dropped from Corollary \ref{MainCorollary}.
\begin{example}
Let the characteristic of $K$ be a prime number $p$, and let $G = \F_p$ be the cyclic group of order $p$. When we look at the additive action of $\F_p$ on $V = K$ via $(\sigma, x) \mapsto \sigma + x$, we see that 
\[K[V]^G = K[x]^G = K[x^p - x] \]
is a polynomial ring. But a non-zero group element $\sigma \in \F_p$ does not have a fixed point, so in particular, $G$ is not a reflection group.
\end{example}

The following example shows that the assumption that $X$ is Cohen-Macaulay cannot be dropped from Corollary \ref{MainCorollary}.
\begin{example}
Let the characteristic of $K$ be $\neq 2$
and consider the affine variety
%\[X = \V (x_1^2 - x_3^2, \, x_2^2 - x_4^2,\, x1x2 + x1x4 - x2x3 - x3x4, \, x1x2 - x1x4 + x2x3 - x3x4 ) \subseteq K^4, \]
\[X = \V (x_1^2 - x_3^2, \, x_2^2 - x_4^2,\, x_1x_2 - x_3x_4, \, x_1x_4 - x_2x_3) \subseteq K^4, \]
which is the union of two planes intersecting at the origin
\[X = \V (x_1 - x_3, \, x_2 - x_4 ) \cup \V (x_1 + x_3, \, x_2 + x_4 ).\]
So Hartshorne's Connectedness Theorem tells us that $X$ is not Cohen-Macaulay at the origin, since it is not connected in codimension 1 there. \\ Now we look at the following representation of the cyclic group of order 2:
\[G = \langle \begin{pmatrix}1&0&0&0\\0&1&0&0\\0&0&-1&0\\0&0&0&-1\end{pmatrix}\rangle \subseteq \GL_4(K)  .\]
This induces an action of $G$ on $X$, which interchanges the two planes. Obviously, this action on $X$ is  not generated by reflections. The invariant ring of the representation $V = K^4$ can be easily seen to be
\[ K[V]^G = K[x_1,\,x_2,\,x_3^2,\, x_3x_4, \, x_4^2].\]
Since we are in a non-modular case, the finite group $G$ is linearly reductive. Therefore, $K[X]^G$ is the quotient ring of $K[V]^G$ modulo the vanishing ideal of $X$.
% \[ K[X]^G = K[V]^G / \mathcal{I}(X) = K[\overline{x}_1,\,\overline{x}_2].\]
We get 
\[ K[X]^G = K[\overline{x}_1,\,\overline{x}_2].\] 
So the invariant ring of the action on $X$ is a polynomial ring, while the group is not generated by reflections.
\end{example}

%-------------------------------
%Literaturverzeichnis
%-------------------------------

\bibliographystyle{plain}
\bibliography{literature}

\end{document}